\documentclass{amsart}
\usepackage[english]{babel}
\usepackage{amssymb,enumerate,bbm,amsmath}
\usepackage[colorlinks=true,linkcolor=blue,citecolor=blue,urlcolor=blue]{hyper ref}
\numberwithin{equation}{section}

\newtheorem{theo}{Theorem}
\newtheorem{pro}[theo]{Proposition}
\newtheorem{lem}[theo]{Lemma}
\newtheorem{cor}[theo]{Corollary}
\newtheorem{conj}[theo]{Conjecture}

\newtheorem{rem}[theo]{Remark}

\newtheorem*{thmA}{Theorem A}
\newtheorem*{thmB}{Theorem B}

\renewcommand{\(}{\left(}
\renewcommand{\)}{\right)}

\renewcommand{\-}{\overline}

\newcommand{\R}{\mathbb{R}}

\renewcommand{\S}{\mathbb{S}}
\renewcommand{\H}{\mathbb{H}}

\renewcommand{\a}{\alpha}

\newcommand{\g}{\gamma}
\renewcommand{\d}{\delta}
\newcommand{\e}{\varepsilon}
\renewcommand{\k}{\kappa}
\renewcommand{\l}{\lambda}

\newcommand{\D}{\Delta}

\renewcommand{\t}{\theta}
\newcommand{\s}{\sigma}

\newcommand{\G}{\Gamma}

\renewcommand{\L}{\Lambda}

\newcommand{\ra}{\rightarrow}

\newcommand{\mrm}{\mathrm}

\begin{document}
\title[Geometric inequalities for hypersurfaces]{Geometric inequalities for hypersurfaces with nonnegative sectional curvature in hyperbolic space}
\author{Yingxiang Hu, Haizhong Li}
\address{Yau Mathematical Sciences Center \\ Tsinghua University \\ Beijing 100086 \\ China\\}
\email{yxhu@math.tsinghua.edu.cn}
\address{Department of Mathematical Sciences \\ Tsinghua University \\ Beijing, 100084 \\ China\\}
\email{hli@math.tsinghua.edu.cn}

\begin{abstract}
In this article, we will use inverse mean curvature flow to establish an optimal Sobolev-type inequality for hypersurfaces $\Sigma$ with nonnegative sectional curvature in $\H^n$. As an application, we prove the hyperbolic Alexandrov-Fenchel inequalities for hypersurfaces with nonnegative sectional curvature in $\H^n$:
\begin{align*}
\int_{\Sigma} p_{2k}\geq \omega_{n-1}\left[\(\frac{|\Sigma|}{\omega_{n-1}}\)^\frac{1}{k}+\(\frac{|\Sigma|}{\omega_{n-1}}\)^{\frac{1}{k}\frac{n-1-2k}{n-1}}\right]^k,
\end{align*}
where $p_i$ is the normalized $i$-th mean curvature. Equality holds if and only if $\Sigma$ is a geodesic sphere in $\H^n$. For a domain $\Omega\subset \H^n$ with $\Sigma=\partial \Omega$ having nonnegative sectional curvature, we prove an optimal inequality for quermassintegral in $\H^n$:
\begin{align*}
W_{2k+1}(\Omega)\geq \frac{\omega_{n-1}}{n}\sum_{i=0}^{k}\frac{n-1-2k}{n-1-2i}C_k^i\(\frac{|\Sigma|}{\omega_{n-1}}\)^\frac{n-1-2i}{n-1},
\end{align*}
where $W_i(\Omega)$ is the $i$-th quermassintegral in integral geometry. Equality holds if and only if $\Sigma$ is a geodesic sphere in $\H^n$. All these inequalities was previously proved by Ge, Wang and Wu \cite{Ge-Wang-Wu2014} under the stronger condition that $\Sigma$ is horospherical convex.
\end{abstract}

{\maketitle}
\section{Introduction}
The geometric inequalities for hypersurfaces in hyperbolic space have recently attracted a lot of attentions, and it motivates the investigation of the curvature flow. Different from a hypersurface in $\R^n$, there are four different kinds of convexity for a hypersurface $(\Sigma,g)$ in $\H^n$:
\begin{enumerate}[(1)]
\item {\em (strictly) convex} if $\k_i>0$ for $i=1,\cdots,n-1$;
\item {\em nonnegative Ricci curvature} if $\k_i \(\sum_{j\neq i}\k_j\)\geq n-2$ for $i=1,\cdots,n-1$;
\item {\em nonnegative sectional curvature} if $\k_i\k_j\geq 1$ for  $1\leq i<j\leq n-1$;
\item {\em horospherical convex (h-convex)} if $\k_i\geq 1$ for $i=1,\cdots,n-1$,
\end{enumerate}
where $\k_1,\cdots,\k_{n-1}$ are the principal curvatures of the hypersurface $(\Sigma,g)$ in $\H^n$, respectively. In fact, these convexity conditions are in strictly ascending order \cite{Alexander-Currier1990,Alexander-Currier1993}. In \cite{Ge-Wang-Wu2014}, Ge, Wang and Wu investigated the $k$-th Gauss-Bonnet curvature $L_k$ on hypersurface $(\Sigma,g)$ in $\H^n$, which is defined by
\begin{align}\label{1.1}
L_k:=\frac{1}{2^k}\d_{j_1 j_2 \cdots j_{2k-1} j_{2k}}^{i_1 i_2 \cdots i_{2k-1} i_{2k}} R_{i_1 i_2}{}^{j_1 j_2} \cdots R_{i_{2k-1}i_{2k}}{}^{j_{2k-1}j_{2k}},
\end{align}
where $R_{ij}{}^{kl}$ is the Riemannian curvature tensor in the local coordinates with respect to the metric $g$, and the generalized Kronecker delta is defined by
\begin{align}\label{1.2}
\d_{i_1i_2\cdots i_r}^{j_1j_2\cdots j_r}=\det\(\begin{matrix}
     \d_{i_1}^{j_1} & \d_{i_1}^{j_2} & \cdots & \d_{i_1}^{j_r} \\
     \d_{i_2}^{j_1} & \d_{i_2}^{j_2} & \cdots & \d_{i_2}^{j_r} \\
      \vdots        &    \vdots      & \vdots & \vdots \\
     \d_{i_r}^{j_1} & \d_{i_r}^{j_2} & \cdots & \d_{i_r}^{j_r}
     \end{matrix}\).
\end{align}
For any hypersurface $(\Sigma,g)$ in $\H^n$, the Gauss-Bonnet curvature $L_k$ of the induced metric of the hypersurface can be expressed by
\begin{align}\label{1.3}
L_k(g)=C_{n-1}^{2k}(2k)!\sum_{j=0}^{k}(-1)^{j}C_k^jp_{2k-2j},
\end{align}
where $p_k$ is the (normalized) $k$-th mean curvature of $\Sigma$, which is defined in (\ref{2.2-2}). Ge, Wang and Wu \cite{Ge-Wang-Wu2014} established an optimal Sobolev-type inequality for h-convex hypersurfaces in $\H^n$.
\begin{thmA}
Let $n\geq 3$ and $2k<n-1$. Any h-convex hypersurface $(\Sigma,g)$ in $\H^n$ satisfies
\begin{align}\label{1.4}
\int_{\Sigma}L_k \geq C_{n-1}^{2k}(2k)!\omega_{n-1}^{\frac{2k}{n-1}}|\Sigma|^{\frac{n-1-2k}{n-1}}.
\end{align}
The equality holds if and only if $\Sigma$ is a geodesic sphere in $\H^n$.
\end{thmA}

When $k=1$, it was proved by Li-Wei-Xiong in \cite{Li-Wei-Xiong2014} that (\ref{1.4}) holds even for any star-shaped and strictly $2$-convex hypersurfaces in $\H^{n}$, i.e., $p_1>0$ and $p_2>0$. It was proposed by Ge-Wang-Wu (see Remark 3.4 in \cite{Ge-Wang-Wu2013}) that whether or not the inequality (\ref{1.4}) still holds for hypersurfaces with nonnegative sectional curvature in $\H^n$. The main purpose of this paper is to give an affirmative answer to this question.
\begin{theo}\label{theo-1}
Let $n\geq 3$ and $2k<n-1$. Any hypersurface $(\Sigma,g)$ with nonnegative sectional curvature in $\H^n$ satisfies
\begin{align}\label{1.5}
\int_{\Sigma}L_k \geq C_{n-1}^{2k}(2k)!\omega_{n-1}^{\frac{2k}{n-1}}|\Sigma|^{\frac{n-1-2k}{n-1}}.
\end{align}
The equality holds if and only if $\Sigma$ is a geodesic sphere in $\H^n$.
\end{theo}
The importance of the inequality (\ref{1.5}) is that it can be viewed as the bricks of other geometric inequalities. To observe this, it follows from Lemmas 3.2 and 3.3 in \cite{Ge-Wang-Wu2014} that
\begin{align}\label{1.6}
\int_{\Sigma}p_{2k}=&\frac{1}{(2k)!C_{n-1}^{2k}}\sum_{i=0}^{k}C_k^i\int_{\Sigma}L_i, \\
W_{2k+1}(\Omega)=&\frac{1}{(2k)!C_{n-1}^{2k}}\frac{1}{n}\sum_{i=0}^{k}C_k^i \frac{n-1-2k}{n-1-2i}\int_{\Sigma}L_i.
\end{align}
Here $W_{2k+1}(\Omega)$ is $(2k+1)$-st quermassintegrals, which will be defined in (\ref{2.3}).

As a direct application of Theorem \ref{theo-1}, we obtain the hyperbolic Alexandorv-Fenchel inequalities for hypersurfaces with nonnegative sectional curvature in $\H^n$, which was previously proved by Ge-Wang-Wu \cite{Ge-Wang-Wu2014} for h-convex hypersurfaces in $\H^n$.
\begin{theo}\label{theo-2.1}
Let $n\geq 3$ and $2k\leq n-1$. Any hypersurface $\Sigma$ with nonnegative sectional curvature in $\H^n$ satisfies
\begin{align}\label{1.7}
\int_{\Sigma} p_{2k}\geq \omega_{n-1}\left[\(\frac{|\Sigma|}{\omega_{n-1}}\)^\frac{1}{k}+\(\frac{|\Sigma|}{\omega_{n-1}}\)^{\frac{1}{k}\frac{n-1-2k}{n-1}}\right]^k,
\end{align}
The equality holds in (\ref{1.7}) if and only if $\Sigma$ is a geodesic sphere in $\H^n$.
\end{theo}

We also obtain the following optimal inequalities for $W_k(\Omega)$ for general odd $k$ in terms of the area $|\Sigma|$.
\begin{theo}\label{theo-2.2}
Let $n\geq 3$ and $2k\leq n-1$. If $\Omega\subset \H^n$ is a domain with smooth boundary $\Sigma=\partial \Omega$ having nonnegative sectional curvature, then
\begin{align}\label{1.8}
W_{2k+1}(\Omega)\geq \frac{\omega_{n-1}}{n}\sum_{i=0}^{k}\frac{n-1-2k}{n-1-2i}C_k^i\(\frac{|\Sigma|}{\omega_{n-1}}\)^\frac{n-1-2i}{n-1}.
\end{align}
The equality holds in (\ref{1.8}) if and only if $\Sigma$ is a geodesic sphere in $\H^n$.
\end{theo}

As a corollary, we solve the isoperimetric problem for hypersurfaces with nonnegative sectional curvature in hyperbolic space with fixed $W_1=\frac{|\Sigma|}{n}$, which was posed by Gao-Hug-Schneider \cite{Gao-Hug-Schneider2003}.
\begin{cor}
Let $n\geq 3$ and $2k\leq n-1$. In the class of hypersurfaces with nonnegative sectional curvature and the fixed $W_1$ in $\H^n$, the minimum of $W_{2k+1}$ is achieved if and only if $\Sigma$ is a geodesic sphere in $\H^n$..
\end{cor}

In order to prove Theorem \ref{theo-1}, the idea is similar to Ge-Wang-Wu \cite{Ge-Wang-Wu2014}. We consider the functional:
\begin{align}\label{1.9}
Q_k(t):=|\Sigma_t|^{-\frac{n-1-2k}{n-1}}\int_{\Sigma_t} L_k.
\end{align}
Let $X_0:M^{n-1}\ra \H^n$ be a smooth embedding such that the initial hypersurface $\Sigma=X_0(M)$ is a closed smooth hypersurface in $\H^n$. We consider the smooth family of immersions $X:M^{n-1} \times [0,T)\ra \H^n$ evolves along the inverse mean curvature flow (IMCF):
\begin{align}\label{1.10}\left\{ \begin{aligned}
\frac{\partial}{\partial t}X(x,t)=&\frac{1}{H(x,t)}\nu(x,t),\\
X(\cdot,0)=&X_0(\cdot),\end{aligned}\right.
\end{align}
where $H(x,t)$ is the mean curvature and $\nu(x,t)$ is the unit outward normal vector of $\Sigma_t=X(M,t)$, respectively.

In Section 3, we will show that if the initial hypersurface $\Sigma$ is smooth, closed and has nonnegative sectional curvature, then the flow hypersurface $\Sigma_t$ of the IMCF has nonnegative sectional curvature for any time $t>0$. This will be crucial in establishing the inequality (\ref{1.5}) for hypersurfaces with nonnegative sectional curvature in hyperbolic space. The method we used here is motivated by the recent important work of Andrews-Chen-Wei \cite{Andrews-Chen-Wei2018}, where they proved that the nonnegativity of sectional curvature is preserved along volume preserving curvature flows in hyperbolic space.

In Section 4, we use Andrews' maximum principle for tensors to show that if the initial hypersurface $\Sigma$ in hyperbolic space is h-convex, then the flow hypersurface $\Sigma_t$ of the IMCF becomes strictly h-convex for $t>0$. The idea we used here follows from the recent work of Andrews and Wei \cite{Andrews-Wei2017}, and the proof does not rely on the constant rank theorem as in \cite{Wang-Xia2014}. This property will be crucial to establish the rigidity part of the inequality (\ref{1.5}) under the weaker condition that $\Sigma$ has nonnegative sectional curvature.

In Section 5, we will show that the functional $Q_k$ is non-increasing along the IMCF, provided that the initial hypersurface has nonnegative sectional curvature. An important observation is that if $Q_k$ is constant along the IMCF, then the hypersurface with nonnegative sectional curvature is h-convex.

In Section 6, by the convergence result of Gerhardt \cite{Gerhardt2011} on the IMCF, we show that the flow approaches to hypersurfaces whose induced metrics belong to the conformal class of the standard round sphere metric. A generalized Sobolev inequality of Guan-Wang \cite{Guan-Wang2004} shows that
\begin{align*}
Q_k(0) \geq \lim_{t\ra \infty} Q_k(t) \geq C_{n-1}^{2k} (2k)! \omega_{n-1}^{\frac{2k}{n-1}}.
\end{align*}
The argument to establish this inequality is the same as that in Ge-Wang-Wu \cite{Ge-Wang-Wu2014}. However, as we use the IMCF instead of the inverse curvature flows used by Ge-Wang-Wu \cite{Ge-Wang-Wu2014}, the rigidity part of the inequality (\ref{1.5}) needs to be proved in a completely different way. If the equality holds in (\ref{1.5}), then $Q_k$ is constant along the IMCF and thus the initial hypersurface is h-convex. For any $t>0$, the flow hypersurface $\Sigma_t$ of IMCF is strictly h-convex. Together with the equality characterization of the inequality (\ref{4.2}), we show that $\Sigma_t$ is totally umbilical and hence it is a geodesic sphere. Finally, the initial hypersurface is smoothly approximated by a family of geodesic spheres, and it must be a geodesic sphere in $\H^n$.

In Section 7, we can also prove the Alexandrov-Fenchel type inequality for $\int_\Sigma p_1$ for hypersurfaces with nonnegative Ricci curvature in $\H^n$. We also mention a weaker version of the Alexandrov-Fenchel type inequality for $\int_\Sigma p_1$, which holds for any star-shaped and mean convex hypersurfaces in $\H^n$.

By studying the quermassintegral preserving flow, Wang and Xia \cite{Wang-Xia2014} proved the Alexandrov-Fenchel type inequalities for h-convex hypersurface in hyperbolic space. More recently, Andrews-Chen-Wei \cite{Andrews-Chen-Wei2018} considered a volume preserving flow with nonnegative sectional curvature, and proved some Alexandrov-Fenchel type inequalities under the weaker assumption of nonnegative sectional curvature.

\section*{Acknowledgements}
We would like to thank Yong Wei for his interest and comments. The second author was supported by NSFC grant No.11671224.

\section{Preliminaries}
\subsection{Curvature integrals and Quermassintegrals}
We recall some basic concepts and formulas in integral geometry. We refer to Santal\'o's book \cite{Santalo1976}, see also Schnerder \cite{Schnerder1993} or Solanes \cite{Solanes2003,Solanes2006} for details.

The hyperbolic space $\H^n$ is an $n$-dimensional simply connected Riemannian manifold of constant sectional curvature $-1$. Let $\Omega$ be a domain with smooth boundary $\Sigma=\partial\Omega$ in $\H^n$, then $\Sigma$ is a closed hypersurface in $\H^n$ with unit outward normal $\nu$. The second fundamental form $h$ of $\Sigma$ is defined by
\begin{align*}
h(X,Y)=\langle \-\nabla_X \nu, Y\rangle,
\end{align*}
for any tangent vector fields $X,Y$ on $\Sigma$. For an orthonormal basis $\{e_1,\cdots,e_{n-1}\}$ of $\Sigma$, the second fundamental form is $h=(h_{ij})$ and the Weingarten tensor is $\mathcal{W}=(h_i^j)=(g^{jk}h_{ki})$, where $g$ is the induced metric on $\Sigma$. The principal curvatures $\k=(\k_1,\cdots,\k_{n-1})$ are the eigenvalues of $\mathcal{W}$.

Let $\s_k$ be the $k$-th elementary symmetric function $\s_k:\R^{n-1}\ra \R$ defined by
\begin{align}\label{2.1}
\s_k(\l)=\sum_{i_1<\cdots<i_k}\l_{i_1}\cdots \l_{i_k}, \quad \text{for}~ \l=(\l_1,\cdots,\l_{n-1})\in \R^{n-1}.
\end{align}
We also take $\s_0=1$ by convention. The Garding cone is defined as
\begin{align*}
\G_k^+=\{\l\in \R^{n-1}~|~\s_j(\l)>0,\forall j\leq k\}.
\end{align*}
We denoted by $\-{\G_{k}^{+}}$ the closure of $\G_{k}^{+}$. Let
$$
p_k(\l)=p_k(\l)=\frac{\s_k(\l)}{C_{n-1}^k}
$$
be the normalized $k$-th symmetric functions, then we have the well-known Newton-MacLaurin inequalities (see Lemma 2.7 on p55 in \cite{Guan2014}).
\begin{lem}
Let $1\leq k\leq n-1$. For $\l \in \G_k^+$, we have
\begin{align}\label{2.2}
p_1 p_{k-1}\geq p_k, \quad p_1 \geq p_2^{1/2} \geq \cdots \geq p_k^{1/k}.
\end{align}
Moreover, the above equalities hold if and only if $\l=\a(1,\cdots,1)$ for some $\a>0$.
\end{lem}

The {\em normalized $k$-th mean curvature} of $\Sigma$ is defined by
\begin{align}\label{2.2-2}
p_k(x)=p_k(\k(x)), \quad x\in \Sigma, \quad  0\leq k \leq n-1,
\end{align}
and the {\em curvature integrals} are defined by
\begin{align*}
V_{n-1-j}(\Omega)=\int_{\Sigma}p_j, \quad 0\leq j\leq n-1.
\end{align*}
For a convex domain $\Omega\subset \H^n$, the {\em quermassintegrals} are defined by
\begin{align}\label{2.3}
W_r(\Omega):=\frac{(n-r)\omega_{r-1}\cdots\omega_{0}}{n\omega_{n-2}\cdots \omega_{n-r-1}}\int_{\mathcal{L}_r}\chi(L\cap \Omega)dL, \quad r=1,\cdots,n-1,
\end{align}
where $\mathcal{L}_r$ is the space of $r$-dimensional totally geodesic subspaces $L$ in $\H^n$, and $dL$ is the natural measure on $\mathcal{L}_r$ which is invariant under the isometry group of $\H^n$. The function $\chi$ is defined to be $1$ if $L\cap \Omega\neq \emptyset$ and to be $0$ otherwise. Furthermore, we set $W_0(\Omega)=\mrm{Vol}(\Omega)$, $W_n(\Omega)=\omega_{n-1}/n$. The quermassintegrals and curvature integrals in $\H^n$ are related by the following recursive formulas (see Proposition 7 in \cite{Solanes2006}):
\begin{align}\label{2.4}
V_{n-1-j}(\Omega)=n\(W_{j+1}(\Omega)+\frac{j}{n-j+1}W_{j-1}(\Omega)\),\quad j=1,\cdots,n-1.
\end{align}

\subsection{Evolution equations for IMCF}
Let $X_0:M^{n-1}\ra \H^n$ be a smooth embedding such that $\Sigma=X_0(M)$ is a closed smooth hypersurface in $\H^n$. We consider the smooth family of immersions $X:M^{n-1} \times [0,T)\ra \H^{n}$ evolves along the inverse mean curvature flow (IMCF):
\begin{align}\label{9.1}\left\{ \begin{aligned}
\frac{\partial}{\partial t}X(x,t)=&\frac{1}{H(x,t)}\nu(x,t),\\
X(\cdot,0)=&X_0(\cdot),\end{aligned}\right.
\end{align}
where $H(x,t)$ is the mean curvature and $\nu(x,t)$ is the unit outward normal vector of the hypersurface $\Sigma_t=X(M,t)$, respectively.

Along the IMCF (\ref{9.1}), we have the following evolution equations on the Weingarten tensor $\mathcal{W}=(h_i^j)$ of $M_t$ (see \cite{Andrews1994}):
\begin{align}\label{9.2}
\frac{\partial}{\partial t}h_i^j=\frac{1}{H^2}\D h_{i}^{j}-\frac{2}{H^3}\nabla_i H\nabla^j H+\frac{1}{H^2}(|A|^2+n-1)h_i^j-\frac{2}{H}(h^2)_i^j,
\end{align}
where $\nabla$ denotes the Levi-Civita connection with respect to the induced metric $g_{ij}$ on $M_t$. For simplicity, we will omit the subscript $t$ and the volume form $d\mu_t$ if there is no ambiguity. By the variational formula by Reilly \cite{Reilly1973}, one can check that along the IMCF we have
\begin{align}\label{2.12}
\frac{d}{dt}\int_{\Sigma}p_k=\int_{\Sigma} \( (n-1-k)p_{k+1}+k p_{k-1}\)\frac{1}{(n-1)p_1}, \quad k=0,\cdots,n-1.
\end{align}

\section{Preserving nonnegative sectional curvature along IMCF}
In this section, we prove that if the initial hypersurface $\Sigma$ in hyperbolic space is smooth, closed and has nonnegative sectional curvature, then the solution $\Sigma_t$ of the IMCF has nonnegative sectional curvature for any time $t>0$. Inspired by a recent work due to Andrews-Chen-Wei \cite{Andrews-Chen-Wei2018}, we show that the nonnegativity of sectional curvature of the hypersurface is preserved along the IMCF. The argument is related to that used by Andrews \cite{Andrews2007} to prove a generalized tensor maximum principle, see Theorem \ref{theo-4.1} in Section 4. However, it can not be deduced directly from that result. The argument combines the ideas of the generalized tensor maximum principle with those of vector bundle maximum principles for reaction-diffusion equations \cite{Andrews-Hopper2011,Hamilton1986}.
\begin{theo}\label{theo-3}
If the initial hypersurface $\Sigma$ has nonnegative sectional curvature, then along the IMCF (\ref{9.1}) the evolving hypersurface $\Sigma_t$ has nonnegative sectional curvature for $t>0$.
\end{theo}
\begin{proof}
The sectional curvature defines a smooth function on the Grassmannian bundle of $2$-dimensional subspace of $TM$. For convenience, we lift this to a function on the orthonormal frame bundle $O(M)$ over $M$: Given a point $x\in M$ and $t\geq 0$, and a frame $\mathbb{O}=\{e_1,\cdots,e_{n-1}\}$ of $T_x M$ which is orthonormal with respect to the metric $g(x,t)$, we define
\begin{align*}
G(x,t,\mathbb{O})=h_{(x,t)}(e_1,e_1) h_{(x,t)}(e_2,e_2)-h_{(x,t)}(e_1,e_2)^2-1.
\end{align*}
We consider a point $(x_0,t_0)\in M\times [0,t_0]$ and a frame $\mathbb{O}_{0}=\{\-{e}_1,\cdots,\-{e}_{n-1}\}$ at which a new minimum of the function $G$ is attained, so that we have
\begin{align*}
G(x,t,\mathbb{O}) \geq G(x_0,t_0,\mathbb{O}_{0}),
\end{align*}
for all $x\in M$, $t\in [0,t_0]$ and $\mathbb{O}\in F(M)_{(x,t)}$. The fact that $\mathbb{O}_0$ achieves the minimum of $G$ over the fiber $F(M)_{(x_0,t_0)}$ implies that $\-e_1$ and $\-e_2$ are eigenvectors of $h_{(x_0,t_0)}$ corresponding to the principal curvatures $\k_1$ and $\k_2$, where $\k_1\leq \k_2 \leq \cdots \leq \k_{n-1}$ are the principal curvatures at $(x_0,t_0)$. Since $G$ is invariant under rotation in the subspace orthogonal to $\{\-e_1,\-e_2\}$, we can assume that $h(\-e_i,\-e_i)=\k_i$, $h(\-e_i,\-e_j)=0$ for $i\neq j$.

We derive the evolution equation for $G$ at $(x_0,t_0,\mathbb{O}_0)$. Note that the evolving frame $\mathbb{O}(t)$ for $T_xM$ is defined by
\begin{align*}
\frac{d}{dt}e_i(t)=\frac{1}{H}\mathcal{W}(e_i(t)), \quad e_i(t_0)=\-e_i, \quad \text{for $i=1,\cdots,n-1$}.
\end{align*}
Then the frame remains orthonormal with respect to the metric $g(x,t)$. From (\ref{9.2}) we get
\begin{align*}
\frac{d}{dt}G |_{(x_0,t_0,\mathbb{O}_0)}=&\k_1 \frac{\partial}{\partial t}h_2^2+\k_2 \frac{\partial}{\partial t}h_1^1 \\
=&\frac{\k_1}{H^2}\D h_{22}+\frac{\k_2}{H^2}\D h_{11}-\frac{2\k_1}{H^3}|\nabla_2 H|^2-\frac{2\k_2}{H^3}|\nabla_1 H|^2\\
 &+\frac{2\k_1\k_2}{H^2}\(|A|^2+n-1-(\k_1+\k_2)H\).
\end{align*}
We denote
\begin{align*}
R_1:=&\frac{\k_1}{H^2}\D h_{22}+\frac{\k_2}{H^2}\D h_{11}-\frac{2\k_1}{H^3}|\nabla_2 H|^2-\frac{2\k_2}{H^3}|\nabla_1 H|^2,\\
R_2:=&\frac{2\k_1\k_2}{H^2}\(|A|^2+n-1-(\k_1+\k_2)H\).
\end{align*}
First, we have
\begin{align*}
\frac{2}{H^2}\(|A|^2-(\k_1+\k_2)H+(n-1)\k_1\k_2\)=&\frac{2}{H^2}\sum_{k=1}^{n-1}(\k_k-\k_2)(\k_k-\k_1) \geq 0;\\
-\frac{2(\k_1\k_2-1)}{H^2}(-|A|^2+(\k_1+\k_2)H)=&-G\frac{2}{H^2}\sum_{k=1}^{n-1} \k_k(-\k_k+\k_1+\k_2).
\end{align*}
Sum all these terms up, we have
\begin{align*}
R_2=&\frac{2\k_1\k_2}{H^2}(|A|^2+n-1-(\k_1+\k_2)H) \\
=&\frac{2}{H^2}\left[ |A|^2-(\k_1+\k_2)H+(n-1)\k_1\k_2+(\k_1\k_2-1)|A|^2 \right.\\
&\left.-(\k_1\k_2-1)(\k_1+\k_2)H\right] \\
=&\frac{2}{H^2}\sum_{k=1}^{n-1}(\k_k-\k_2)(\k_k-\k_1)-G\frac{2}{H^2}\sum_{k=1}^{n-1} \k_k(-\k_k+\k_1+\k_2) \\
\geq &-G\frac{2}{H^2}\sum_{k=1}^{n-1} \k_k(-\k_k+\k_1+\k_2) \\
\geq &-\frac{2(n-2)}{n-1}G=:-CG,
\end{align*}
where the last inequality follows from
\begin{align*}
 &\frac{2}{H^2}\sum_{k=1}^{n-1} \k_k(-\k_k+\k_1+\k_2)\\
=&\frac{2}{H^2}(-|A|^2+(\k_1+\k_2)H)\\
\leq & \frac{2}{H^2}\(-\frac{1}{n-1}H^2+H^2\)\leq \frac{2(n-2)}{n-1}.
\end{align*}
To estimate $R_1$, we consider the second derivatives of $G$ along a curve on $O(M)$ defined as follows: Let $\g$ be any geodesic of $g(t_0,\cdot)$ in $M$ with $\g(0)=x_0$, and define a frame $\mathbb{O}(s)=(e_1(s),\cdots,e_{n-1}(s))$ at $\g(s)$ by taking $e_i(0)=\-e_i$ for each $i$ and $\nabla_s e_i(s)=\G_{ij} e_{j}(s)$ for some constant antisymmetric matrix $\G$. Then we have
\begin{equation}\label{9.4}
\begin{split}
\left. \frac{d^2}{ds^2}G(x(s),t_0,\mathbb{O}(s)) \right|_{s=0}=&\k_2 \nabla_s^2 h_{11}+\k_1 \nabla_s^2 h_{22}+2(\nabla_s h_{22} \nabla_s h_{11}-(\nabla_s h_{12})^2) \\
&+4\sum_{p=3}^{n-1}\G_{1p}\k_2 \nabla_s h_{1p}+4\sum_{p=3}^{n-1}\G_{2p}\k_1\nabla_s h_{2p} \\
&+2\sum_{p=3}^{n-1}\G_{1p}^2 \k_2(\k_p-\k_1)+2\sum_{p=3}^{n-1}\G_{2p}^2\k_1(\k_p-\k_2).
\end{split}
\end{equation}
Since $G$ has a minimum at $(x_0,t_0,\mathbb{O}_0)$, the RHS of (\ref{9.4}) is nonnegative for any choice of $\G$. Minimizing all $\G$ gives
\begin{align*}
0\leq &\k_2 \nabla_s^2 h_{11}+\k_1 \nabla_s^2 h_{22}+2\(\nabla_s h_{22}\nabla_s h_{11}-(\nabla_s h_{12})^2\)\\
&-2\sum_{p=3}^{n-1} \frac{\k_2}{\k_p-\k_1}(\nabla_s h_{1p})^2-2\sum_{p=3}^{n-1} \frac{\k_1}{\k_p-\k_2}(\nabla_s h_{2p})^2,
\end{align*}
where the terms on the last line vanishes if the denominators vanish, since the corresponding component of $\nabla h$ vanishes in that case. This gives
\begin{equation}\label{9.5}
\begin{split}
R_1 \geq &-\frac{2}{H^3}(\k_1|\nabla_2 H|^2+\k_2|\nabla_1 H|^2)-\frac{2}{H^2}\sum_{k=1}^{n-1}\left[ \nabla_kh_{22}\nabla_k h_{11}-(\nabla_k h_{12})^2\right]\\
&+\frac{2}{H^2}\sum_{k=1}^{n-1}\sum_{p=3}^{n-1}\left[\frac{\k_2}{\k_p-\k_1}(\nabla_k h_{1p})^2+\frac{\k_1}{\k_p-\k_2}(\nabla_k h_{2p})^2 \right].
\end{split}
\end{equation}
From $\nabla_i G=0$, we have
\begin{align}\label{9.6}
\k_2 \nabla_i h_{11}+\k_1 \nabla_i h_{22}=0, \quad i=1,\cdots,n-1.
\end{align}
Then we get
\begin{align*}
\nabla_k h_{22}\nabla_k h_{11}-(\nabla_k h_{12})^2=-\frac{\k_2}{\k_1}(\nabla_k h_{11})^2-(\nabla_k h_{12})^2 \leq 0.
\end{align*}
So the second term in RHS of (\ref{9.5}) is nonnegative. The third term in RHS of (\ref{9.5}) is also nonnegative. In order to show $R_1\geq 0$, it suffices to show
\begin{align}\label{9.7}
-\frac{2\k_2}{H^3}|\nabla_1 H|^2-\frac{2}{H^2}(\nabla_1 h_{22}\nabla_1 h_{11}-(\nabla_1 h_{22})^2)+\frac{2}{H^2}\sum_{p=3}^{n-1}\frac{\k_2}{\k_p-\k_1}(\nabla_1 h_{pp})^2\geq 0,
\end{align}
and
\begin{align}\label{9.8}
-\frac{2\k_1}{H^3}|\nabla_2 H|^2-\frac{2}{H^2}(\nabla_2 h_{22}\nabla_2 h_{11}-(\nabla_2 h_{11})^2)+\frac{2}{H^2}\sum_{p=3}^{n-1}\frac{\k_1}{\k_p-\k_2}(\nabla_2 h_{pp})^2\geq 0.
\end{align}
Here we only prove (\ref{9.7}), and (\ref{9.8}) can be proved similarly. If $n-1=2$, then the third term in LHS of (\ref{9.7}) vanishes, and hence
\begin{align*}
&-\frac{2\k_2}{H^3}|\nabla_1 H|^2-\frac{2}{H^2}(\nabla_1 h_{22}\nabla_1 h_{11}-(\nabla_1 h_{22})^2)\\
\geq &-\frac{2\k_2}{H^3}(\nabla_1 h_{11}+\nabla_1 h_{22})^2-\frac{2}{H^2}(\nabla_1 h_{22}\nabla_1 h_{11}-(\nabla_1 h_{22})^2) \\
=& -\frac{2\k_2}{H^3}(\nabla_1 h_{11})^2\frac{(\k_2-\k_1)^2}{\k_1^2}+\frac{2}{H^2}(\nabla_1 h_{11})^2\frac{\k_2(\k_1+\k_2)}{\k_1^2} \\
=& \frac{2\k_2}{H^3}(\nabla_1 h_{11})^2\frac{(\k_2+\k_1)^2-(\k_2-\k_1)^2}{\k_1^2} \geq 0.
\end{align*}
If $n-1\geq 3$, let $\hat{H}:=H-\k_1-\k_2$, then we have $\hat{H}\geq (n-3)\k_1>0$. By the Cauchy-Schwarz inequality,
\begin{align*}
 &\sum_{p=3}^{n-1}\frac{(\nabla_1 h_{pp})^2}{\k_p-\k_1} \cdot \(\hat{H}-(n-3)\k_1\)\\
=&\sum_{p=3}^{n-1}\frac{(\nabla_1 h_{pp})^2}{\k_p-\k_1} \cdot \sum_{p=3}^{n-1}(\k_p-\k_1) \\
\geq &\(\sum_{p=3}^{n}|\nabla_1 h_{pp}|\)^2 \geq |\nabla_1 \hat{H}|^2,
\end{align*}
we have
\begin{align}\label{9.9}
\sum_{p=3}^{n}\frac{(\nabla_1 h_{pp})^2}{\k_p-\k_1}\geq &\frac{|\nabla_1 \hat{H}|^2}{\hat{H}-(n-3)\k_1}.
\end{align}
It follows from (\ref{9.6}) and (\ref{9.9}) that
\begin{equation}\label{9.10}
\begin{split}
&-\frac{2\k_2}{H^3}|\nabla_1 H|^2-\frac{2}{H^2}(\nabla_1 h_{22}\nabla_1 h_{11}-(\nabla_1 h_{22})^2)+\frac{2}{H^2}\sum_{p=3}^{n-1}\frac{\k_2}{\k_p-\k_1}(\nabla_1 h_{pp})^2\\
\geq &\frac{2\k_2}{H^2}\left[-\frac{1}{H}\(\frac{\k_1-\k_2}{\k_1}\nabla_1 h_{11}+\nabla_1 \hat{H}\)^2+\frac{\k_1+\k_2}{\k_1^2}(\nabla_1 h_{11})^2+\frac{|\nabla_1 \hat{H}|^2}{\hat{H}-(n-3)\k_1}\right] \\
= &\frac{2\k_2}{H^2} \left[ \(-\frac{(\k_1-\k_2)^2}{\k_1^2 H}(1+\e)+\frac{\k_1+\k_2}{\k_1^2}\)(\nabla_1 h_{11})^2\right. \\
 &\left.+\(-\frac{1+\e^{-1}}{H}+\frac{1}{\hat{H}-(n-3)\k_1}\)|\nabla_1 \hat{H}|^2 \right],
\end{split}
\end{equation}
where we have used the inequality
$$
(a+b)^2 \leq (1+\e) a^2+(1+\e^{-1}) b^2, \quad \text{for any $\e>0$}.
$$
If $\k_1=\k_2$, then (\ref{9.7}) follows from
\begin{align*}
 &-\frac{2\k_2}{H^3}|\nabla_1 H|^2-\frac{2}{H^2}(\nabla_1 h_{22}\nabla_1 h_{11}-(\nabla_1 h_{22})^2)+\frac{2}{H^2}\sum_{p=3}^{n-1}\frac{\k_2}{\k_p-\k_1}(\nabla_1 h_{pp})^2 \\
\geq &\frac{2\k_1}{H^2}\(-\frac{1}{H}+\frac{1}{\hat{H}-(n-3)\k_1}\)|\nabla_1 \hat{H}|^2\\
= &\frac{2(n-1)\k_1^2}{H^3(\hat{H}-(n-3)\k_1)}|\nabla_1 \hat{H}|^2\geq 0.
\end{align*}
If $\k_1<\k_2$, we take $\e=\frac{\hat{H}(\k_1+\k_2)+4\k_1\k_2}{(\k_2-\k_1)^2}$. Then we have
\begin{align*}
-\frac{(\k_1-\k_2)^2}{\k_1^2 H}(1+\e)+\frac{\k_1+\k_2}{\k_1^2}=0,
\end{align*}
and
\begin{align*}
 &-\frac{1+\e^{-1}}{H}+\frac{1}{\hat{H}-(n-3)\k_1}\\
=&\frac{(n-2)\k_1+\k_2}{H(\hat{H}-(n-3)\k_1)}-\frac{(\k_2-\k_1)^2}{H(\hat{H}(\k_1+\k_2)+4\k_1\k_2)}\\
=&\frac{((n-2)\k_1+\k_2)(\hat{H}(\k_1+\k_2)+4\k_1\k_2)-(\k_2-\k_1)^2(\hat{H}-(n-3)\k_1)}{H(\hat{H}-(n-3)\k_1)(\hat{H}(\k_1+\k_2)+4\k_1\k_2)}\\
=&\frac{\left[((n-2)\k_1+\k_2)(\k_1+\k_2)-(\k_2-\k_1)^2\right]\hat{H}}{H(\hat{H}-(n-3)\k_1)(\hat{H}(\k_1+\k_2)+4\k_1\k_2)}\\
 &+\frac{4\k_1\k_2((n-2)\k_1+\k_2)+(n-3)\k_1(\k_2-\k_1)^2}{H(\hat{H}-(n-3)\k_1)(\hat{H}(\k_1+\k_2)+4\k_1\k_2)}\\
\geq &0.
\end{align*}
From this, we show that (\ref{9.7}) holds. A similar argument shows that (\ref{9.8}) also holds and we obtain $R_1\geq 0$. Finally, we conclude that $\frac{\partial}{\partial t}G \geq -C G$ at a spatial minimum point, and hence the maximum principle (see Lemma 3.5 in \cite{Hamilton1986}) gives
\begin{align*}
\inf_{x\in M,\mathbb{O}\in F(M)_{(x,t)}}G(x,t,\mathbb{O}) \geq e^{-Ct} \inf_{x\in M,\mathbb{O}\in F(M)_{(x,0)}}G(x,0,\mathbb{O})\geq 0.
\end{align*}
along the IMCF, which finishes the proof of Theorem \ref{theo-3}.
\end{proof}

Combining Theorem \ref{theo-3} with the result of Gerhardt \cite{Gerhardt2011}, we obtain the following proposition.
\begin{pro}\label{pro-convergence-result}
If the initial hypersurface $\Sigma$ has nonnegative sectional curvature, then along the IMCF (\ref{9.1}) the flow hypersurface $\Sigma_t$ has nonnegative sectional curvature for $t>0$. Moreover, the hypersurfaces $\Sigma_t$ become more and more umbilical in the sense of
\begin{align}\label{9.11}
|h_i^j-\d_i^j| \leq C e^{-\frac{t}{n-1}}, \quad t>0,
\end{align}
i.e., the principal curvatures are uniformly bounded and converge exponentially fast to $1$.
\end{pro}

\section{Preserving h-convexity along IMCF}
In this section, we prove that if the initial hypersurface $\Sigma$ in hyperbolic space is h-convex, then the solution $\Sigma_t$ of the IMCF is strictly h-convex for any time $t>0$. We recall the maximum principle for tensors, which was first proved by Hamilton \cite{Hamilton1982} and was generalized by Andrews \cite{Andrews2007}.
\begin{theo}[\cite{Andrews2007}]\label{theo-4.1}
Let $S_{ij}$ be a smooth time-varying symmetric tensor field on a closed manifold $M$, satisfying
\begin{align}\label{3.1}
\frac{\partial}{\partial t} S_{ij} =a^{kl}\nabla_k \nabla_l S_{ij}+u^k \nabla_k S_{ij}+N_{ij},
\end{align}
where $a^{kl}$ and $u^k$ are smooth, $\nabla$ is a (possibly time-dependent) smooth symmetric connection, and $a^{kl}$ is positive definite everywhere. Suppose that
\begin{align}\label{3.2}
N_{ij} v^i v^j +\sup_{\L} 2 a^{kl}(2\L_k^p \nabla_l S_{ip} v^i -\L_k^p \L_l^q S_{pq}) \geq 0,
\end{align}
where $S_{ij}\geq 0$ and $S_{ij} v^j=0$. If $S_{ij}$ is positive definite everywhere on $M$ at $t=0$, then it is positive definite on $M\times [0,T]$.
\end{theo}

Without resorting to the constant rank theorem, here we follow the spirit of Andrews-Wei \cite{Andrews-Wei2017}, and prove that the h-convexity
is preserved along the IMCF.

\begin{theo}\label{theo-4.2}
Let $\Sigma_t$, $t\in [0,T)$ be a solution of IMCF (\ref{9.1}). If the initial hypersurface $\Sigma$ is h-convex, then the evolving hypersurface $\Sigma_t$ is strictly h-convex for $t\in (0,T)$.
\end{theo}
\begin{proof}
We show that the h-convexity is preserved along the IMCF, and it becomes strictly h-convex for $t>0$. We take $S_{ij}:=h_{i}^j-\d_i^j$. Then the $h$-convexity is equivalent to $S_{ij}\geq 0$. By (\ref{9.2}), the tensor $S_{ij}$ evolves by
\begin{equation}\label{Sij-evolution}
\begin{split}
\frac{\partial}{\partial t}S_{ij}=&\frac{1}{H^2}\D S_{ij}-\frac{2}{H^3}\nabla_i H  \nabla^j H+(|A|^2+n-1)S_{ij} \\
&-\frac{2}{H}(S_{ik}S_{kj}+2S_{ij})+\frac{1}{H^2}(|A|^2+(n-1)-2H)\d_{i}^{j}.
\end{split}
\end{equation}
To apply the tensor maximum principle, we need to show that (\ref{3.2}) holds provided that $S_{ij}\geq 0$ and $S_{ij}v^{j}=0$. Let $(x_0,t_0)$ be the point where $S_{ij}$ has a null vector $v$. By continuity, we can assume that $h_i^j$ has all eigenvalues distinct and in increasing order at $(x_0,t_0)$, that is $\k_{n-1}>\k_{n-2}>\cdots>\k_1$. The null eigenvector condition $S_{ij}v^j=0$ implies that $v=e_1$ and $S_{11}=\k_1-1=0$ at $(x_0,t_0)$. The terms in (\ref{Sij-evolution}) which contains $S_{ij}$ and $S_{ik}S_{kj}$ satisfies the null vector condition. For the last term in (\ref{Sij-evolution}) we have
\begin{align*}
\frac{1}{H^2}(|A|^2+(n-1)-2H) \geq \frac{1}{H^2}\(\frac{H^2}{n-1}+(n-1)-2H\)\geq 0.
\end{align*}
Thus, it remains to show that
\begin{align*}
Q_1:=-\frac{2}{H^3}|\nabla_1 H|^2+2\sup_{\L} \frac{\d^k_l}{H^2}(2\L_k^p\nabla_l S_{1p}-\L_k^p\L_l^q S_{pq}) \geq 0.
\end{align*}
Note that $S_{11}=0$ and $\nabla_k S_{11}=0$ at $(x_0,t_0)$, the supremum over $\L$ can be explicitly computed as follows.
\begin{align*}
 &\frac{2\d^k_l}{H^2}(2\L_k^p\nabla_l S_{1p}-\L_k^p\L_l^q S_{pq})\\
=&\frac{2}{H^2}\sum_{k=1}^{n-1}\sum_{p=2}^{n-1} (2\L_k^p \nabla_k S_{1p}-(\L_k^p)^2 S_{pp}) \\
=&\frac{2}{H^2}\sum_{k=1}^{n-1}\sum_{p=2}^{n-1}\left[\frac{(\nabla_k S_{1p})^2}{S_{pp}}-\(\L_k^p-\frac{\nabla_k S_{1p}}{S_{pp}}\)^2 S_{pp}\right].
\end{align*}
Thus the supremum is obtained by taking $\L_k^p=\frac{\nabla_k S_{1p}}{S_{pp}}$. The required inequality for $Q_1$ becomes
\begin{align*}
Q_1=-\frac{2}{H^3}|\nabla_1 H|^2+\frac{2}{H^2}\sum_{k=1}^{n-1}\sum_{p=2}^{n-1}\frac{(\nabla_k S_{1p})^2}{S_{pp}}\geq 0.
\end{align*}
By the Codazzi equation we have $\nabla_1 S_{1p}=\nabla_1 h_{1p}=\nabla_p h_{11}=0$ at $(x_0,t_0)$, we have
\begin{align*}
\sum_{k>1} \frac{1}{\k_k} (\nabla_1 h_{kk})^2 \cdot \sum_{k>1}\k_k \geq \(\sum_{k>1}|\nabla_1 h_{kk}|\)^2 \geq \left|\nabla_1 \sum_{k>1}h_{kk}\right|^2=|\nabla_1 H|^2,
\end{align*}
which gives $\frac{|\nabla_1 H|^2}{H} \leq \sum_{k>1} \frac{1}{\k_k} (\nabla_1 h_{kk})^2$. Hence, we get
\begin{align*}
Q_1=&-\frac{2}{H^3}|\nabla_1 H|^2+\frac{2}{H^2}\sum_{k>1,l>1}\frac{1}{\k_l-1}(\nabla_1 h_{kl})^2\\
\geq &-\frac{2}{H^2}\sum_{k>1}\frac{1}{\k_k}(\nabla_1 h_{kk})^2+\frac{2}{H^2}\sum_{k>1,l>1}\frac{1}{\k_l-1}(\nabla_1 h_{kl})^2 \\
\geq &\frac{2}{H^2}\sum_{k>1,l>1}\(\frac{1}{\k_l-1}-\frac{1}{\k_l}\)(\nabla_1 h_{kl})^2 \geq 0.
\end{align*}
Thus, the Andrews' maximum principle (Theorem \ref{theo-4.1}) implies that the $h$-convexity is preserved along the IMCF.

Finally, we show that $\Sigma_t$ is strictly h-convex for $t>0$. If this is not true, then there exists some interior point $(x_0,t_0)$ such that the smallest principal curvature is $1$. By the strong maximum principle, there exists a parallel vector field $v$ such that $S_{ij}v^iv^j=0$ on $\Sigma_{t_0}$. Then the smallest principal curvature is $1$ on $\Sigma_{t_0}$ everywhere. This contradicts with the fact that on any closed hypersurface in $\H^{n}$, there exists at least one point where all the principal curvatures are strictly larger than one. This completes the proof.
\end{proof}

\section{Monotonicity formula}
In this section, we prove the monotonicity of functional $Q_k$ along the IMCF. Let
$$
\~L_k=\sum_{i=0}^{k}C_k^i(-1)^ip_{2k-2i}=\frac{1}{(2k)!C_{n-1}^{2k}}L_k, \quad \~N_k=\sum_{i=0}^{k}C_k^i(-1)^ip_{2k+1-2i}.
$$
The variational formula for $\int_{\Sigma}\~L_k$ is the following.
\begin{lem}\label{lem-monotonicity}
\begin{align}\label{4.1}
\frac{d}{dt}\int_{\Sigma}\~L_k =\frac{n-1-2k}{n-1}\int_{\Sigma}\~L_k +\frac{n-1-2k}{n-1}\int_{\Sigma}\(\frac{\~N_k}{p_1}-\~L_k\).
\end{align}
\end{lem}
\begin{proof}
It follows from (\ref{2.12}) that along the IMCF, we have
\begin{align*}
  & \frac{d}{dt}\int_{\Sigma} \~L_k \\
= & \int_{\Sigma} \sum_{i=0}^k C_k^i(-1)^i\left[ (n-1-2k+2i)p_{2k-2i+1}+(2k-2i)p_{2k-2i-1}\right]\frac{1}{(n-1)p_1}\\
= & \frac{n-1-2k}{n-1}\int_{\Sigma} \frac{1}{p_1}\sum_{i=0}^k C_k^i (-1)^i p_{2k-2i+1} \\
 &+\frac{1}{n-1}\int_{\Sigma} \frac{1}{p_1}\sum_{i=1}^k(-1)^i \left[ C_k^i (2i)-C_k^{i-1} (2k-2i+2)\right]p_{2k-2i+1} \\
=& \frac{n-1-2k}{n-1}\int_{\Sigma} \frac{\~N_k}{p_1} \\
=& \frac{n-1-2k}{n-1}\int_{\Sigma}\~L_k +\frac{n-1-2k}{n-1}\int_{\Sigma}\(\frac{\~N_k}{p_1}-\~L_k\).
\end{align*}
\end{proof}

To show that the monotonicity of the functional $Q_k(t)$ along the IMCF, we need to show that
$$
\frac{\~N_k}{p_1}-\~L_k \leq 0.
$$
For this purpose, we need the following lemma. Define the cone
$$
\G:=\{\l\in\R^{n-1}~|~\l_i\l_j \geq 1, \forall i\neq j\},
$$
then it is easy to see the cone $\{\l\in\R^{n-1}~|~\l_i\geq 1\}$ is strictly contained in $\G$.
\begin{lem}\label{lem-key-inequality}
For any $\k\in \G$, we have
\begin{align}\label{4.2}
\~N_k - p_1 \~L_k \leq 0.
\end{align}
The equality holds if and only if one of the following two cases holds:
\begin{enumerate}[(i)]
\item $\k_i=\k_j$ for all i,j;
\item if $k\geq 2$, there exists one $i$ with $\k_i>1$ and $\k_j=1$ for all $j\neq i$.
\end{enumerate}
\end{lem}
\begin{proof}
The proof relies on a crucial observation due to Ge-Wang-Wu (see Lemma 4.3 in \cite{Ge-Wang-Wu2014}): the inequality (\ref{4.2}) is equivalent to the following inequality:
\begin{align}\label{4.3}
\sum_{\substack{1\leq i_m \leq n-1\\ i_{p}\neq i_{q}(p\neq q)}}\k_{i_1}(\k_{i_2}\k_{i_3}-1)(\k_{i_4}\k_{i_5}-1)\cdots (\k_{i_{2k-2}}\k_{i_{2k-1}}-1)(\k_{i_{2k}}-\k_{i_{2k+1}})^2 \geq 0,
\end{align}
where the summation takes over all $(2k+1)$-elements permutation of $\{1,2,\cdots,n-1\}$. For $\k\in \G$, we have $\k_i>0$ for all $i$ and
\begin{align*}
\k_{i_p}\k_{i_q}-1 \geq 0,\quad (\k_{i_{p}}-\k_{i_{q}})^2\geq 0,
\end{align*}
for all distinct $p,q$. It follows that if $\k \in \G$ then (\ref{4.3}) holds, and hence the inequality (\ref{4.2}) holds. Now we analyze the equality case in (\ref{4.3}). If $k=1$, then the equality in (\ref{4.3}) reduces to
\begin{align*}
\sum_{\substack{1\leq i_m \leq n-1\\ i_{p}\neq i_{q}(p\neq q)}}\k_{i_1}(\k_{i_{2}}-\k_{i_{3}})^2 = 0.
\end{align*}
Since $\k_i>0$ for all $i$, we get $\k_i=\k_j$ for all $i,j$. If $k\geq 2$, then the equality in (\ref{4.3}) implies that one of the following two cases holds:
\begin{enumerate}[($i'$)]
\item  $\k_i=\k_j$ for all $i,j$;
\item  There exists at least one $\k_i$ which is distinct with another $\k_j$, $j\neq i$.
\end{enumerate}
Without loss of generality, we may assume that $\k_1 \leq \cdots \leq \k_{n-1}$. We claim that $\k_1\geq 1$. If not, then it follows from $\k\in \G$ that
\begin{align*}
0<\k_1<1<\frac{1}{\k_1}\leq \k_2 \leq \cdots \leq \k_{n-1}.
\end{align*}
Then we have $\k_i\k_j >1$ for all distinct $i,j\in \{2,\cdots,n-1\}$. We have
\begin{align*}
\k_1(\k_{i_2}\k_{i_3}-1)(\k_{i_4}\k_{i_5}-1)\cdots (\k_{i_{2k-2}}\k_{i_{2k-1}}-1)(\k_{i_{2k}}-\k_{i_{2k+1}})^2=0,
\end{align*}
where $(i_2,\cdots,i_{2k+1})$ is taken over all $2k$-elements permutation of $\{2,\cdots,n-1\}$. This implies that $\k_i=\k_j$ for all $i,j\in \{2,\cdots,n-1\}$. We also have
\begin{align*}
\k_{i_1}(\k_{i_2}\k_{i_3}-1)(\k_{i_4}\k_{i_5}-1)\cdots (\k_{i_{2k-2}}\k_{i_{2k-1}}-1)(\k_2-\k_1)^2=0,
\end{align*}
where $(i_1,\cdots,i_{2k-1})$ is taken over all $(2k-1)$-elements permutation of $\{3,\cdots,n-1\}$. Together with $\k_i=\k_j$ for all $i,j\in \{2,\cdots,n-1\}$, we get $\k_2=\cdots=\k_{n-1}=1$. It follows that $\k_1 \k_2=\k_1<1$, which contradicts with $\k\in \G$. Therefore, we have $\k_1\geq 1$.

Now we claim that if ($ii'$) holds, then
\begin{align}\label{4.4}
1=\k_1=\cdots=\k_{n-2}<\k_{n-1}.
\end{align}
To show (\ref{4.4}), without loss of generality we may assume $\k_i=\k_{n-1}$ in ($ii'$), then $\k_1<\k_{n-1}$. Together with $\k_1\k_{n-1}\geq 1$, we have $1<\k_{n-1}$. We have
\begin{align}\label{4.5}
(\k_{i_2}\k_{i_3}-1)(\k_{i_4}\k_{i_5}-1)\cdots (\k_{i_{2k-2}}\k_{i_{2k-1}}-1)(\k_{i_{2k}}-\k_{i_{2k+1}})^2=0,
\end{align}
where $(i_2,\cdots,i_{2k+1})$ is taken over all $(2k)$-elements permutation of $\{1,\cdots,n-1\}$. Now we prove (\ref{4.4}) by induction.
If $\k_1>1$, then $\k_i\k_j>1$ for all $i,j\in \{1,\cdots,n-1\}$. If we take $i_{2k+1}=n-1$, then it follows from (\ref{4.5}) that
$$
1<\k_1<\k_2=\k_3=\cdots=\k_{n-1}.
$$
However, in this case we have
\begin{align*}
(\k_{i_2}\k_{i_3}-1)(\k_{i_4}\k_{i_5}-1)\cdots (\k_{i_{2k-2}}\k_{i_{2k-1}}-1)(\k_{n-1}-\k_{1})^2>0,
\end{align*}
where $(i_2,\cdots,i_{2k-1})$ is taken over all $(2k-2)$-elements permutation of $\{2,\cdots,n-2\}$. This contradicts with (\ref{4.5}), so we have $\k_1=1$. Assume that we have proved
$$
1=\k_1=\cdots=\k_{j} \leq \k_{j+1}\leq \cdots \leq \k_{n-1},
$$
and we need to show that $\k_{j+1}=1$. If $\k_{j+1}>1$, then we take $i_{2m}=m$, where $m=1,\cdots,j$ and $i_{2k+1}=n-1$. Let $(i_3,i_5,\cdots,i_{2j+1},i_{2j+2},\cdots,i_{2k})$ is taken over all $(2k-1+j)$-elements permutation of $\{j+1,j+2,\cdots,n-2\}$. As we have $\k_{i_{2m}}\k_{i_{2m+1}}\geq \k_{1} \k_{j+1}>1$ for $m=1,\cdots,j$, it follows from (\ref{4.5}) that
$$
1<\k_1=\cdots=\k_j<\k_{j+1}=\cdots=\k_{n-1}.
$$
However, in this case we may take $i_{2m}=m+1$, where $m=1,\cdots,j-1$, $i_{2k}=1$ and $i_{2k+1}=n-1$. Let $(i_3,i_5,\cdots,i_{2j-1},i_{2j},i_{2j+1},\cdots,i_{2k-1})$ is taken over all $(2k-1+j)$-elements permutation of $\{j+1,j+2,\cdots,n-2\}$, then we have $\k_{i_{2m}}\k_{i_{2m+1}}\geq \k_{2} \k_{j+1}>1$ for $m=1,\cdots,j-1$ and $\k_{i_{2m}}\k_{2m+1}>1$ for $m=j,\cdots,k-1$. Thus we have
\begin{align*}
(\k_{i_2}\k_{i_3}-1)(\k_{i_4}\k_{i_5}-1)\cdots (\k_{i_{2k-2}}\k_{i_{2k-1}}-1)(\k_{n-1}-\k_{1})^2>0,
\end{align*}
which contradicts with (\ref{4.5}). Hence we conclude that $\k_{j+1}=1$. Induction on $j$ then verifies (\ref{4.4}), which completes the proof.
\end{proof}
\begin{rem}
For $\k\in \G$, either (i) or (ii) implies that $\k\in \{\l\in\R^{n-1}~|~\l_i\geq 1\}$. As a consequence, on a hypersurface $\Sigma$ with nonnegative sectional curvature, if $\~N_k - p_1 \~L_k \equiv 0$, then $\Sigma$ is h-convex.
\end{rem}

\begin{theo}\label{theo-4}
The functional $Q_k(t)$ is non-increasing along the IMCF (\ref{9.1}), provided that the initial hypersurface has nonnegative sectional curvature.
\end{theo}
\begin{proof}
By the variational formula (\ref{4.1}) of $\int_{\Sigma}\~L_k$ and (\ref{4.2}), we have
\begin{align}\label{4.6}
\frac{d}{dt}\int_{\Sigma}L_k \leq \frac{n-1-2k}{n-1}\int_{\Sigma}L_k.
\end{align}
On the other hand, by (\ref{2.12}) for $k=0$, we have
\begin{align}\label{4.7}
\frac{d}{dt}|\Sigma_t|=|\Sigma_t|.
\end{align}
From (\ref{5.17}), we know that the quantity $Q_k(t)$ is positive along the IMCF. Combining with (\ref{4.6}) and (\ref{4.7}), we obtain
\begin{align}\label{4.8}
\frac{d}{dt}Q_k(t) \leq 0.
\end{align}
\end{proof}

\section{Proof of Main Theorems}
In this section, we will give the proof of Theorem \ref{theo-1}.
\begin{proof}[Proof of Theorem \ref{theo-1}]
The proof follows the spirit of the proof of Theorem A due to Ge-Wang-Wu \cite{Ge-Wang-Wu2014}, the main difference is that here we use the inverse mean curvature flow and we only assume the initial hypersurface $\Sigma$ has nonnegative sectional curvature. First, by the definition of $Q_k(t)$, it is equivalent to show that
\begin{align}\label{5.1}
Q_k(0) \geq C_{n-1}^{2k} (2k)! \omega_{n-1}^{\frac{2k}{n-1}}.
\end{align}
By Proposition \ref{pro-convergence-result}, the nonnegativity of sectional curvature of $\Sigma_t$ is preserved along the IMCF. It follows from Theorem \ref{theo-4} that the functional $Q_k(t)$ is monotone non-increasing. Hence, it suffices to show
\begin{align}\label{5.2}
\lim_{t\ra \infty} Q_k(t) \geq C_{n-1}^{2k}(2k)! \omega_{n-1}^{\frac{2k}{n-1}}.
\end{align}
The hyperbolic space can be represented as a warped product $\H^n=\R^{+}\times \S^{n-1}$ endowed with the metric
$$
\-g= dr^2+\l(r)^2 g_{\S^{n-1}},
$$
where the warping function $\l(r)=\sinh r$ and $g_{\S^{n-1}}$ is the standard round metric on $\S^{n-1}$. Since $\Sigma$ has nonnegative sectional curvature, it is a strictly convex hypersurface in $\H^n$. By the Hadamard theorem for strictly convex hypersurface in $\H^n$ (see do Carmo and Warner \cite{doCarmo-Warner1970} or Theorem 10.3.1 in Gerhardt's book \cite{Gerhardt2006}), $\Sigma$ bounds a strictly convex body $\Omega$ in $\H^n$, so it can be written as a graph of function $r(\t)$, $\t\in \S^{n-1}$. Hence the initial hypersurface $\Sigma$ can be represented by an embedding $X_0:\S^{n-1}\ra \H^n$. Let $X(t,\cdot):\S^{n-1} \ra \H^n$, $t\in [0,T)$ be the solution of IMCF with the initial data given by $X_0$. The evolving hypersurface $\Sigma_t$ can be written as a graph
$$
\Sigma_t=\{ (r(t,\t),\t) ~:~\t \in \S^{n-1} \},
$$
where $r(t,\cdot)$ is a positive function defined on $\S^{n-1}$.
We define a new function $\varphi:\S^{n-1}\ra \R$ by
$$
\varphi(\t)=\Phi(r(\t)),
$$
where $\Phi$ is a positive function satisfying $\Phi'(r)=\frac{1}{\l(r)}$. Let $\varphi_i=\nabla_i \varphi$ and $\nabla_{ij}\varphi=\nabla_j\nabla_i \varphi$ denote the covariant derivatives of $\varphi$ with respect to $g_{\S^{n-1}}$. We also define another function
$$
v=\sqrt{1+|\nabla \varphi|^2_{g_{\S^{n-1}}}}.
$$
Let $g=(g_{ij})$ be the induced metric on $\Sigma$ and $h=(h_{ij})$ be the second fundamental form in terms of the coordinates $\t^j$. Then we have
\begin{align}\label{5.3}
g_{ij}=\l^2(\s_{ij}+\varphi_i\varphi_j), \quad h_{ij}=\frac{\l}{v}(\l'(\s_{ij}+\varphi_i\varphi_j)-\varphi_{ij}).
\end{align}
By (\ref{9.11}) in Proposition \ref{pro-convergence-result}, we have
\begin{align}\label{5.4}
\l=O(e^{\frac{t}{n-1}}), \quad |\nabla \varphi|_{g_{\S^{n-1}}}+|\nabla^2 \varphi|_{g_{\S^{n-1}}}=O(e^{-\frac{t}{n-1}}).
\end{align}
Then by the identity $\l'^2=\l^2+1$ we get
\begin{align}\label{5.5}
\l'=\l+\frac{1}{2\l}+O(e^{-\frac{4t}{n-1}}).
\end{align}
We also have
\begin{align}\label{5.6}
\frac{1}{v}=1-\frac{1}{2}|\nabla \varphi|_{g_{\S^{n-1}}}^2 + O(e^{-\frac{4t}{n-1}}).
\end{align}
By (\ref{5.3}), the Weingarten tensor $\mathcal{W}=(h_i^j)$ of $\Sigma_t$ can be expressed as
\begin{equation}\label{5.7}
\begin{split}
h_{i}^j=&\frac{\l'}{v\l}(\d_i^j-\frac{\varphi_i^j}{\l'}+\frac{\varphi_i\varphi_j\varphi^{jl}}{v^2 \l'}) \\
=&\d_i^j+\(\frac{1}{2\l^2}-\frac{1}{2}|\nabla \varphi|_{g_{\S^{n-1}}}^2\)\d_i^j-\frac{\varphi_i^j}{\l}+O(e^{-\frac{4t}{n-1}}).
\end{split}
\end{equation}
We take
\begin{align}\label{5.8}
T_i^j:=\(\frac{1}{2\l^2}-\frac{1}{2}|\nabla \varphi|_{g_{\S^{n-1}}}^2\)\d_i^j-\frac{\varphi_i^j}{\l},
\end{align}
then the Gauss equation gives
\begin{equation}\label{5.9}
\begin{split}
R_{ij}{}^{kl}=&-(\d_i^k \d_j^l-\d_i^l \d_j^k)+(h_i^kh_j^l-h_i^lh_j^k) \\
=&\d_i^k T_j^l+T_i^k \d_j^l -T_i^l \d_j^k-\d_i^l T_j^k +O(e^{-\frac{4t}{n-1}}).
\end{split}
\end{equation}
A direct calculation gives the expression for $L_k$:
\begin{align}\label{5.10}
L_k=2^k k! (n-1-k) \cdots (n-2k)\s_k(T)+O(e^{-\frac{(2k+2)t}{n-1}}).
\end{align}
By a similar argument as the proof of Ge-Wang-Wu \cite{Ge-Wang-Wu2014}, we can get
\begin{align}\label{5.17}
\lim_{t\ra \infty} |\Sigma_t|^{-\frac{n-1-2k}{n-1}}\int_{\Sigma_t}L_k \geq (2k)!C_{n-1}^{2k} \omega_{n-1}^{\frac{2k}{n-1}}.
\end{align}
When (\ref{5.1}) is an equality, then $Q_k(t)\equiv (2k)!C_{n-1}^{2k} \omega_{n-1}^{\frac{2k}{n-1}}$, which implies that the equality in (\ref{4.2}) on $\Sigma_t$. For $k\geq 2$, it may not be totally umbilical everywhere. However, in both cases (i) and (ii) of Lemma \ref{lem-key-inequality}, the nonnegativity of sectional curvature implies the h-convexity. Thus, the initial hypersurface $\Sigma$ is h-convex. It follows from Theorem \ref{theo-4.2} that $\Sigma_t$ is strictly h-convex for $t>0$, which excludes the case (ii) in Lemma \ref{lem-key-inequality}. Thus, we have $\Sigma_t$ is totally umbilical for $t>0$ and hence it is a geodesic sphere in $\H^n$. As $t\ra 0$, the initial hypersurface $\Sigma$ is smoothly approximated by a family of geodesic spheres, and thus it is also a geodesic sphere in $\H^n$. It is easy to see that if $\Sigma$ is a geodesic sphere of radius $r$, then the area of $\Sigma$ is $|\Sigma|=\omega_{n-1}\sinh^{n-1}r$. By (\ref{1.3}), the integral of $L_k$ is
\begin{align*}
\int_{\Sigma}L_k=&\omega_{n-1}\sinh^{n-1}r \cdot C_{n-1}^{2k}(2k)!\sum_{j=0}^{k}(-1)^j C_k^j (\coth r)^{2k-2j} \\
=&\omega_{n-1}\sinh^{n-1}r \cdot  C_{n-1}^{2k}(2k)!(-1+\coth^2 r)^{k} \\
=&C_{n-1}^{2k}(2k)! \omega_{n-1}\sinh^{n-1-2k}r \\
=&C_{n-1}^{2k}(2k)! \omega_{n-1}^{\frac{2k}{n-1}}|\Sigma|^{\frac{n-1-2k}{n-1}}.
\end{align*}
Hence the equality holds in (\ref{1.4}) on a geodesic sphere. This completes the proof of Theorem \ref{theo-1}.
\end{proof}

\begin{proof}[Proof of Theorems \ref{theo-2.1} and \ref{theo-2.2}]
By Theorem \ref{theo-1}, for $2k<n-1$ we have
\begin{align*}
\int_{\Sigma}L_k \geq  C_{n-1}^{2k}(2k)!\omega_{n-1}^{\frac{2k}{n-1}}|\Sigma|^{\frac{n-1-2k}{n-1}}.
\end{align*}
Together with the expressions (\ref{1.6}) for $\int_{\Sigma}p_{2k}$ and $W_{2k+1}$, we prove (\ref{1.7}) and (\ref{1.8}) for $2k<n-1$. By Theorem \ref{theo-1}, the equality holds in (\ref{1.7}) or (\ref{1.8}) if and only if $\Sigma$ is a geodesic sphere.

When $2k+1=n$, the Hadamard theorem for strictly convex hypersurfaces in $\H^n$ implies that $\Sigma$ is diffeomorphic to $\S^{n-1}$. By the Gauss-Bonnet-Chern theorem \cite{Chern1944,Chern1945}, we have the identity:
\begin{align}
\int_{\Sigma} L_{\frac{n-1}{2}}=(n-1)!\omega_{n-1}.
\end{align}
Thus, (\ref{1.7}) and (\ref{1.8}) also hold for $2k+1=n$. By Theorem \ref{theo-1}, the equality case in (\ref{1.7}) or (\ref{1.8}) for $2k+1=n$ if and only if $\Sigma$ is a geodesic sphere.
\end{proof}

\section{Other Alexandrov-Fenchel type inequality}
In this section, we show that the Alexandrov-Fenchel inequality for $\int_{\Sigma}p_1$ on hypersurfaces with nonnegative Ricci curvature in $\H^n$. This inequality has been proved under the assumption of h-convexity by Ge-Wang-Wu \cite{Ge-Wang-Wu2014} with a help of a result of Cheng-Zhou \cite{Cheng-Zhou2014} (see also Wang-Xia \cite{Wang-Xia2014} by using the quermassintegral preserving flows in hyperbolic space).
\begin{theo}\label{theo-5}
Let $n\geq 3$. Any hypersurface $\Sigma$ with nonnegative Ricci curvature in $\H^n$ satisfies
\begin{align}\label{6.1}
\int_{\Sigma} p_1 \geq \omega_{n-1} \left[\(\frac{|\Sigma|}{\omega_{n-1}}\)^2 + \(\frac{|\Sigma|}{\omega_{n-1}}\)^{\frac{2(n-2)}{n-1}} \right]^\frac{1}{2}.
\end{align}
The equality holds if and only if $\Sigma$ is a geodesic sphere in $\H^n$.
\end{theo}
\begin{proof}
The nonnegativity of the Ricci curvature implies that the hypersurface is strictly convex, and the Hadamard theorem for strictly convex hypersurfaces in $\H^n$ implies that $\Sigma$ is $2$-convex and starshaped. It follows from the result of Li-Wei-Xiong \cite{Li-Wei-Xiong2014} that
\begin{align}\label{6.3}
\int_{\Sigma}p_2 \geq |\Sigma|+\omega_{n-1}^\frac{2}{n-1} |\Sigma|^\frac{n-3}{n-1}.
\end{align}
Moreover, the equality holds in (\ref{6.3}) if and only if $\Sigma$ is a geodesic sphere. On the other hand, if the hypersurface $\Sigma$ has nonnegative Ricci curvature, by (1.4) in \cite{Cheng-Zhou2014} and a direct calculation, we have
\begin{align}\label{6.2}
\int_{\Sigma} p_2 \int_\Sigma p_0 \leq \(\int_\Sigma p_1\)^2.
\end{align}
Together with (\ref{6.3}), we get (\ref{6.1}). If the equality holds in (\ref{6.1}), then (\ref{6.3}) is also an equality, and hence $\Sigma$ is a geodesic sphere.
\end{proof}
It remains an open problem that whether or not the nonnegativity of Ricci curvature is preserved along the IMCF. There are many attempts to prove (\ref{6.1}) under weaker assumptions. The first author \cite{Hu2018} proved the Willmore inequality for hypersurface in hyperbolic space:
\begin{thmB}
Let $n\geq 3$. Any starshaped and mean-convex hypersurface $\Sigma$ in $\H^n$ satisfies
\begin{align*}
\int_{\Sigma}p_{1}^2 \geq |\Sigma|+\omega_{n-1}^{\frac{2}{n-1}}|\Sigma|^\frac{n-3}{n-1}.
\end{align*}
The equality holds if and only if $\Sigma$ is a geodesic sphere in $\H^n$.
\end{thmB}

Motivated by Theorem \ref{theo-2.1}, we would like to propose the following conjecture.
\begin{conj}
Let $n\geq 3$ and $2k+1\leq n-1$. Any hypersurface $\Sigma$ with nonnegative sectional curvature in $\H^n$ satisfies
\begin{align}\label{6.4}
\int_{\Sigma} p_{2k+1}\geq \omega_{n-1}\left[\(\frac{|\Sigma|}{\omega_{n-1}}\)^\frac{2}{2k+1}+\(\frac{|\Sigma|}{\omega_{n-1}}\)^{\frac{2}{2k+1}\frac{n-2-2k}{n-1}}\right]^\frac{2k+1}{2},
\end{align}
The equality holds in (\ref{6.4}) if and only if $\Sigma$ is a geodesic sphere in $\H^n$.
\end{conj}
The inequality (\ref{6.4}) was proved by Wang-Xia \cite{Wang-Xia2014} under the stronger condition that $\Sigma$ is h-convex.

\end{document}